[1994/12/01]% LaTeX date must December 1994 or later
\documentclass{amsart}
\usepackage{mathptmx}

\usepackage[english]{babel}
\usepackage[utf8]{inputenc}
\usepackage{textcomp}  
\usepackage[T1]{fontenc}
\usepackage{graphicx}
\usepackage[automark]{scrpage2}
\usepackage[a4paper,left=4cm, right=2cm,top=3cm, bottom=3cm]{geometry} 

\usepackage{color}
\usepackage[pagebackref]{hyperref}

\usepackage{amsthm,amsmath,amssymb,amsfonts}
\usepackage{mathrsfs}

\newtheorem{Theorem}{Theorem}[section]
\newtheorem{Lemma}[Theorem]{Lemma}
\newtheorem{Proposition}[Theorem]{Proposition}
\newtheorem{Corollary}[Theorem]{Corollary}

\theoremstyle{definition} 

\newtheorem{Remark}[Theorem]{Remark}
\newtheorem{Definition}[Theorem]{Definition}

\newcommand{\cal}{\mathcal}
\newcommand{\R}{\mathbb{R}}
\newcommand{\C}{\mathbb{C}}
\newcommand{\N}{\mathbb{N}}
\newcommand{\Z}{\mathbb{Z}}
\newcommand{\hyp}{\R^n \setminus \R^l}

\newcommand{\ph}{\varphi}
\newcommand{\Om}{\Omega}
\newcommand{\eps}{\varepsilon}
\newcommand{\Sc}{{\cal S}}

\DeclareMathOperator{\rloc}{rloc}
\DeclareMathOperator{\rinf}{rinf}
\DeclareMathOperator{\tr}{tr}
\DeclareMathOperator{\Ext}{Ext}
\DeclareMathOperator{\dist}{dist}

\newcommand{\BR}[1][s]{B_{p,q}^{#1}(\R^n)}
\newcommand{\FR}[1][s]{F_{p,q}^{#1}(\R^n)}
\newcommand{\Fpp}[2]{F_{p,p}^{#1}(#2)}

\newcommand{\bspq}{\BR}
\newcommand{\fspq}{\FR}

\newcommand{\fO}[1][\Z^{\Om}]{f_{p,q}^{s}(#1)}

\newcommand{\FO}[1][\Om]{F_{p,q}^s(#1)}

\newcommand{\FtBar}[1][\Om]{\tilde{F\,}\!_{p,q}^s(\bar{#1})}

\newcommand{\Ft}[1][\Om]{\tilde{F\,}\!_{p,q}^s(#1)}

\newcommand{\Frinf}[1][\Om]{F_{p,q}^{s,\rinf}(#1)}

\newcommand{\Frloc}[1][\Om]{F_{p,q}^{s,\rloc}(#1)}

\newcommand{\sint}{\lfloor s\rfloor}
\newcommand{\srest}{\{s\}}

\usepackage{amssymb}
\begin{document}



\title{Wavelet decomposition techniques and Hardy inequalities for function spaces on cellular domains}
\author[Benjamin Scharf]{Benjamin Scharf} 
\email{benjamin.scharf@ma.tum.de} 
\address{Technische Universit\"at\\ M\"unchen, Fakult\"at Mathematik\\ Boltzmannstrasse 3\\ D-85748 Garching\\Germany}

\thanks{Benjamin Scharf: Technische Universit\"at M\"unchen, Fakult\"at Mathematik, Boltzmannstrasse 3, D-85748 Garching, Germany, benjamin.scharf@ma.tum.de}

\begin{abstract}
A rather tricky question is the construction of wavelet bases on domains for suitable function spaces (Sobolev, Besov, Triebel-Lizorkin type). In his monograph from 2008, Triebel presented an approach how to construct wavelet (Riesz) bases in function spaces of Besov and Triebel-Lizorkin type on cellular domains, in particular on the cube. However, he had to exclude essential exceptional values of the smoothness parameter $s$, for instance the theorems do not cover the Sobolev space $W_2^1(Q)$ on the $n$-dimensional cube $Q$ for $n$ at least 2. 

Triebel also gave an idea how to deal with those exceptional values for the Triebel-Lizorkin function space scale on the cube $Q$: He suggested to introduce modified function spaces for the critical values, the so-called reinforced spaces. In this paper we start examining these reinforced spaces and transfer the crucial decomposition theorems necessary for establishing a wavelet basis from the non-critical values to analogous results for the critical cases now decomposing the reinforced function spaces of Triebel-Lizorkin type.
\end{abstract}

\maketitle

% \linenumbers


\section{Introduction}
Nowadays the theory and application of wavelet decompositions plays an important role not only for the study of function spaces (of Lebesgue, Hardy, Sobolev, Besov, Triebel-Lizorkin type) but also for its applications in signal and numerical analysis, partial differential equations and image processing.
  
A rather tricky question is the construction of wavelet bases on domains $\Om \subset \R^n$ for suitable function spaces. The main problem is the handling of the boundary faces of the domain. One starting point are the papers of Ciesielski and Figiel \cite{CF83}, \cite{CF83a} and \cite{Cie84} dealing with spline bases for classical Sobolev and Besov spaces on compact $C^{\infty}$ manifolds. Related approaches and extensions are given in \cite{Dah97}, \cite{DS98}, \cite{Dah01}, \cite{Coh03}, \cite{HS04}, \cite{JK07} and \cite{FG08}. 

A major breakthrough is described in the monograph \cite{Tri08} of Triebel for cellular domains. A cellular domain is a disjoint union of diffeomorphic images of a cube. The most prominent example is the unit cube $Q$ in $\R^n$. Furthermore, all $C^{\infty}$-domains are cellular domains.

On the one hand Triebel constructed wavelet (Riesz) frames, not wavelet bases, for Triebel-Lizorkin spaces $\FO[\Om]$ for $C^{\infty}$ domains $\Om$ with natural exceptional values $s-\frac{1}{p} \in \N_0$ in \cite[Theorem 5.27]{Tri08} for general dimensions $n$ and general smoothness parameter $s$.

On the other hand Triebel constructed wavelet (Riesz) basis for $\FO[\Om]$ where $\Om$ is an $n$-dimensional cellular domain. But he had to exclude the exceptional values $s-\frac{k}{p} \notin \N_0$ for $k\in \{1,\ldots,n\}$, see \cite[Theorem 6.30]{Tri08}. For instance, the most prominent Sobolev space $W_2^1(Q)$ is exceptional and upto now there seems to be no construction of a wavelet basis in Triebel's sense for $W_2^1(Q)$, see the overview given in \cite[Section 5.3.1, Remark 5.50]{Tri08}.

A proposal how to deal with these cases is presented in \cite[Section 6.2.4]{Tri08}. At first, one considers the situation for the unit cube $Q$: The idea is to modify the spaces $\FO[Q]$ and to ``reinforce them'', now named $\Frinf[Q]$: One takes an $f \in \FO[Q]$ and for every critical value $l \in \{0,\ldots,n-1\}$, i.\,e.\ when
\begin{align*}
 s-\frac{n-l}{p} \in \N_0,
\end{align*}
one requires $f$ to fulfil the additional reinforce property $R_l^{r,p}$. Roughly speaking, this reinforce property asks for a certain decay of the derivatives of $f$ at the faces (edges, vertices) of dimension $l \in \{0,\ldots,n-1\}$ of the unit cube. The construction of the reinforced Triebel-Lizorkin function spaces $\Frinf[Q]$ ensures that in the non-critical cases the spaces $\Frinf[Q]$ and $\FO[Q]$ coincide. The main aim of thesis \cite{Sch13} is the construction of wavelet (Riesz) basis for the spaces $\Frinf[Q]$ without any exceptional values. 

In this paper which is an excerpt of Chapter 3 of thesis \cite{Sch13} we give an insight into the first necessary decomposition techniques for incorporating the exceptional values. Instead of dealing with boundary faces of dimension $0$ to $n-1$ for the $n$-dimensional cube $Q$  we consider the model case $\Om=\hyp$, where $\R^l$ stands symbolically for an $l$-dimensional plane in $\R^n$. We say: An $f \in S'(\R^n)$ belongs to $\Frinf[\hyp]$ if, and only if, it belongs to $\fspq$ and furthermore, if
\begin{align*}
 s-\frac{n-l}{p} \in \N_0,
\end{align*}
we ask for an additional decay property at the boundary $\R^l$, see Definition \ref{reinforcedhyp}.

The main aim of the paper is the proof of decomposition techniques for the reinforced Triebel-Lizorkin spaces $\Frinf[\hyp]$ similar to \cite[Section 6.1.4]{Tri08}, in particular (6.68) and Theorem 6.23 of \cite{Tri08}.  In Theorem \ref{Zerleger} for the non-exceptional values resp.\ Theorem \ref{Zerlegercrit} for the exceptional values we show that an element of $\Frinf[\hyp]=\FR$ resp.\ $\Frinf[\hyp] \subsetneq \FR$ belongs to the refined localization space $\Frloc[\hyp]$ if, and only if, all possible traces (boundary values) at $\R^l$ are vanishing. This means that an $f \in \Frinf[\hyp]$ can be approximated in $\Frinf[\hyp]$-norm by sequences of smooth functions vanishing at the plane $\R^l$ if, and only if, all meaningful traces at $\R^l$ are vanishing. 

Later on, the decomposition techniques are crucial for deriving a wavelet decomposition for $\Frinf[\hyp]$ in the same way as Triebel did, see \cite[Theorem 6.30]{Tri08}. One can decompose $f \in \Frinf[\hyp]$ into boundary and interior parts. The interior part belongs to the refined localization space $\Frloc[\hyp]$, which admits interior wavelet decompositions by a wavelet basis completely supported away from $\R^l$, while the boundary parts can be decomposed into boundary wavelets using wavelet-friendly extension operators. This is done (for the non-exceptional cases) in \cite[Theorem 6.28]{Tri08} and is the crucial starting point for the wavelet decomposition on the cube $Q$, see  \cite[Section 6.1.7]{Tri08}. This step from the space decomposition to the wavelet decomposition and the transition from $\hyp$ to the cube $Q$ is not part of this paper but will be published in the future. For now the interested might read Chapter 4 of thesis \cite{Sch13}.

There are two main ingredients for the proof of the decomposition techniques: Firstly, we need an alternative characterization of the refined localization spaces $\Frloc[\Om]$, see Proposition \ref{rlocequi}. We prove that $f \in \FO[\Om]$ belongs to $\Frloc[\Om]$ if, and only if, it fulfills a certain decay at the boundary  of $\Om$. This proposition is a generalization of results by Triebel, see \cite[Corollary 5.15]{Tri01} and also \cite[Theorem 2.18]{Tri08}.

Secondly, to use the vanishing traces at the boundary of $\Om$, we will prove Hardy inequalities for $\R^n$-functions at $l$-dimensional planes. Hardy inequalities and similar results are a widely used tool in function spaces. This goes back to the paper of Hardy \cite{Har20} from 1920. A nice overview on the Hardy inequality and recent results is given in the book \cite{KMP07} by Kufner, Maligranda and Persson. In connection with Besov and Triebel-Lizorkin function spaces Hardy inequalities are closely connected to the theory of envelopes, see Haroske \cite{Har10} and Triebel \cite{Tri01}.
In Lemma \ref{HardyZerleger} we will show: \textit{Let $d(x)$ be the distance of $x$ to the plane $\R^l$. Then there is a constant $c>0$ such that 
\begin{align*}
 \|d^{-s}(\cdot) f|L_p(\R^n)\| \leq c \sum_{\underset{|\alpha|=r}{\alpha \in \N_l^n}} \|d^{-s+r}(\cdot) D^{\alpha}f|L_p(\R^n)\|
\end{align*}
for all $f \in C^r(\R^n)$ with $(D^{\beta}f)(x',0)=0$ for all $x' \in \R^l$ and $\beta \in \N_l^n$ with $|\beta|\leq r-1$.}

%%%%%%%%%%%%%%%%%%%

\section{Preliminaries}
Let $\R^n$ be the Euclidean $n$-space, $\Z$ be the set of integers, $\N$ be the set of natural numbers, $\N_0=\N \cup \{0\}$ and $\overline{\N}_0=\N_0 \cup \{\infty\}$. By $|x|$ we denote the usual Euclidean norm of $x \in \R^n$, by $\|x|X\|$ the (quasi)-norm of an element $x$ of a (quasi)-Banach space $X$. If $S \subset \R^n$, then we denote the $n$-dimensional Lebesgue measure of $S$ by $|S|$. 

By $\Sc(\R^n)$ we mean the Schwartz space on $\R^n$, by $\Sc'(\R^n)$ its dual. The Fourier transform of $f \in \Sc'(\R^n)$ resp.\ its inverse will be denoted by $\hat{f}$ resp.\ $\check{f}$. The convolution of $f \in \Sc'(\R^n)$ and $\varphi \in \Sc(\R^n)$ will be denoted by $f * \varphi$. With $supp \ f$ we denote the support of a distribution $f \in \Sc'(\R^n)$.

By $L_p(\R^n)$ for $0<p\leq \infty$ we denote the usual quasi-Banach space of $p$-integrable complex-valued functions with respect to the Lebesgue measure $|\cdot|$ with the usual $\sup$-norm modification for $p=\infty$. If $\Om \subset \R^n$ is open, then we denote by $L_p(\Om)$ the Lebesgue $L_p$-space on $\Om$.

Let $k \in \N_0$. By $C^k(\R^n)$ we denote the space of all functions $f: \R^n \rightarrow \C$ which are $k$-times continuously differentiable (continuous, if $k=0$)
and bounded.

Let $s \in \R$, $0<p\leq \infty$ resp.\ $0<p<\infty$ and $0<q\leq \infty$. By $\bspq$ and $\fspq$ we denote the Besov- and Triebel-Lizorkin function spaces on $\R^n$. 

If $X,Y$ are quasi-Banach spaces, then by the notation $X \hookrightarrow Y$ we mean that $X \subset Y$ and that the inclusion map is bounded. In the text we will usually use the term "norm" also if we only have a quasi-Banach space to deal with.

Let $B_r(x_0)=\{x \in \R^n: |x-x_0|<r\}$ be the open ball with centre $x_0$ and radius $r>0$. Furthermore, we shorten $B_r:=B_r(0)$ and $B:=B_1$. 

Throughout the text all unimportant constants will be called $c,c',C$ etc.\ or we will directly write $A \lesssim B$ which means that there is a constant $C>0$ such that $A \leq C \cdot B$. Only if extra clarity is desirable, the dependency of the parameters will be stated explicitly. The concrete value of these constants may vary in different formulas but remains the same within one chain of inequalities. By $A \sim B$ we mean that there are constants $C_1,C_2>0$ such that $ C_1 \cdot B \leq A \leq C_2 \cdot B$.

\subsection{Basic properties of function spaces $\BR$ and $\FR$}
\begin{Definition}
\label{def:sigma}
 Let $s \in \R, 0<p\leq \infty, 0<q\leq \infty$ and $n$ be the dimension of $\R^n$. Then we define
 \begin{align*}
  \sigma_p := n \cdot \left(\frac{1}{p}-1\right)_+ \text{ and } \sigma_{p,q}:= n \cdot \left(\frac{1}{\min(p,q)-1}\right)_+,
 \end{align*}
 where $a_+=\max(a,0)$. 
 Furthermore, if $s\in \mathbb{R}$, then there are uniquely determined $\sint \in \mathbb{Z}$ and $\srest \in (0,1]$ with $s=\sint+\srest$. 
\end{Definition}

\begin{Proposition}[Homogeneity property of $\FR$]
\label{homogen}
 Let $0<p<\infty$, $0<q\leq \infty$ and $s>\sigma_{p,q}$. Then for all $\lambda \in (0,1]$ and $f \in \FR$ with 
 \begin{align*}
  supp\ f \subset B_{\lambda}=\{x \in \R^n: |x|<\lambda\}
 \end{align*}
 it holds
 \begin{align*}
  \|f(\lambda \cdot)|\FR\| \sim \lambda^{s-\frac{n}{p}} \|f|\FR\|
 \end{align*}
\end{Proposition}
\begin{proof}
 This is a reformulation of \cite[Theorem 2.11]{Tri08} going back to \cite[Corollary 5.16]{Tri01}. 
\end{proof}

Let $l \in \N$, $l<n$ and $1\leq j_1 < \ldots< j_l \leq n$. We set
\begin{align*}
x^{j_1,\ldots,j_l}:=(x_1,\ldots,x_{j_1-1},x_{j_1+1},\ldots,,x_{j_l-1},x_{j_l+1},\ldots,x_n) \in \R^{n-l}
\end{align*}
with obvious modifications if $j_1=1$ or $j_l=n$.
Let $ f: \R^n \rightarrow \C$. Then we define 
\begin{align*}
 f^{x^{j_1,\ldots,j_l}}(x_{j_1},\ldots, x_{j_l}):=f(x_1,\ldots,x_{j_1-1},x_{j_1},x_{j_1+1},\ldots,,x_{j_l-1},x_{j_l},x_{j_l+1},\ldots,x_n)
\end{align*}
as a function on $\R^l$ for a fixed $x^{j_1,\ldots,j_l} \in \R^{n-l}$.

\begin{Proposition}
 \label{Fubini}
 Let $n \geq 2$, $l \in \N$ and $l<n$. Let
 \begin{align*}
 0<p < \infty, 0< q \leq \infty \text{ and } s> \sigma_{p,q}.
 \end{align*}
Then $\FR$ has the Fubini property, i.\,e.\ for all $f \in  \FR$ it holds
 \begin{align}
 \label{Fubeq}
  \|f|\FR\| \sim \sum_{1\leq j_1 < \ldots< j_l \leq n}  \Big\| \big\| f^{x^{j_1,\ldots,j_l}} |\FO[\R^{l}] \big\| | L_p(\R^{n-l}) \Big\|  
 \end{align}
\end{Proposition}
\begin{proof}
 The proof is an application of the $1$-dimensional Fubini property for $\FR$, see \cite[Theorem 4.4]{Tri01}
\end{proof}

\subsection{Function spaces on domains}
\begin{Definition}
\label{defdom}
Let $\Om$ be a domain, i.\,e.\ non-empty open set, in $\R^n$, $\Gamma=\partial \Om$ its boundary and $\overline{\Om}$ its closure. By $D(\Om)$ we denote the set of all functions $f \in D(\R^n)$ with support inside $\Om$ and by $D'(\Om)$ its usual topological dual space. 

Denote by $g|\Om \in D'(\Om)$ the restriction of $g$ to $\Om$, hence $ (g|\Om)(\ph)=g(\ph) \text{ for } \ph \in D(\Om)$. We introduce 
\begin{align*}
 \FO:=\{f \in D'(\Om): f&=g|\Om \text{ for some } g \in \FR\}, \\
     \|f|\FO\|&=\inf \|g|\FR\|,
\end{align*}
where the infimum is taken over all $g\in \FR$ with $g|\Om=f$. Moreover, let
\begin{align*}
\FtBar:=\{f \in \FR: supp \, f \in \overline{\Om} \} 
\end{align*}
with the quasi-norm from $\FR$. Then 
\begin{align*}
\Ft:=\{f \in D'(\Om): f&=g|\Om \text{ for some } g \in \FtBar\}, \\
\|f|\Ft\|&=\inf \|g|\FtBar\|,
\end{align*}
where the infimum is taken over all $g\in \FtBar$ with $g|\Omega=f$. 
\end{Definition}

\begin{Definition}
\label{FrlocDef} We now introduce the refined localization spaces $\Frloc$.
We start with a Whitney decomposition of $\Om$ in the same way as in \cite[Section 2.1.2]{Tri08}. For more details see Stein \cite[Theorem 3, p.\ 16]{Ste70}. Let
\begin{align*}
 Q^0_{l,r} \subset Q^1_{l,r}, \quad l \in \N_0, r =1,\ldots,M_j \text{ with } M_j \in \overline{\N}_0
\end{align*}
be concentric (open) cubes in $\R^n$, sides parallel to the axes of coordinates, centred at $2^{-l}m^r$ for an $m^r \in \Z^n$. The side length of $Q^0_{l,r}$ shall be $2^{-l}$, the side length of $Q^1_{l,r}$ shall be $2^{-l+1}$. We call this collection of cubes a Whitney decomposition of $\Om$ if the cubes $Q^0_{l,r}$ are pairwise disjoint, if
\begin{align*}
 \Om=\bigcup_{l,r} \overline{Q}_{l,r}^{\,0}, \quad \dist(Q^1_{0,r},\Gamma) \gtrsim 1 \quad \text{and} \quad \dist(Q^1_{l,r},\Gamma) \sim 2^{-l} \text{ for } l \in \N . 
\end{align*}
By the construction in \cite[Theorem 3, p.\ 16]{Ste70} one can furthermore assume that for adjacent cubes $Q^0_{l,r}$ and $Q^0_{l',r'}$ it holds $|l-l'|\leq 1$.  

Let $\varrho=\{\varrho_{j,r}\}$ be a suitable resolution of unity for Whitney cubes, i.\,e.\
\begin{align}
\label{resun}
 supp \ \varrho_{j,r} \subset Q_{j,r}^1, \quad \|D^{\alpha} \varrho_{j,r} (x)\| \leq c_{\alpha} 2^{j|\alpha|}, \: x \in \Om, \alpha \in \N_0^n
\end{align}
for some $c_{\alpha}>0$ independent of $x,j,r$ and 
\begin{align*}
 \sum_{j=0}^{\infty} \sum_{r=1}^{M_j} \varrho_{j,r}(x)=1 \text{ if } x \in \Om.
\end{align*}
Let $0\leq p < \infty$, $0<q\leq\infty$ and $s>\sigma_{p,q}$. Then
\begin{align*}
 \Frloc:=\left\{ f \in D'(\Om): \|f| \Frloc\|_{\varrho}< \infty\right\} 
\end{align*}
with
\begin{align*}
 \|f| \Frloc\|_{\varrho}:=\left(\sum_{j=0}^{\infty}\sum_{r=1}^{M_j}\| \varrho_{j,r} f|\FR\|^p\right)^{\frac{1}{p}}. 
\end{align*}
\end{Definition}
\begin{Remark}
\label{rlocDense}
 The definition of $\Frloc$ is independent of the choice of the re\-so\-lution of unity $\varrho$. The space $D(\Om)$ is dense in $\Frloc$ if $0<p,q<\infty$. This follows by the density of $D(\R^n)$ in $\fspq$ and pointwise multiplier arguments.
\end{Remark}

\begin{Remark}
In \cite[Section 2]{Tri08} Triebel introduced (interior) $u$-wavelet systems for $\Om$, $u$-wavelet bases, $u$-Riesz bases  and interior sequence spaces $\fO$ on domains $\Om$. The main result of this section is the wavelet decomposition in \cite[Theorem 2.38]{Tri08} of $\Frloc$:
\end{Remark}

\begin{Remark}
\label{Rem:Ethick}
In \cite[Proposition 3.10]{Tri08} Triebel showed that for the class of bounded $E$-thick domains it holds
\begin{align*}
\Ft=\Frloc.
\end{align*}
This is not valid for $\Om=\hyp$. Actually, we have for $s>0$
\begin{align*}
 \Ft[\hyp] \cong \FR \text{ since } \overline{\Om}=\R^n \text{ and } \Ft[\partial \Om]=\{0\}.
\end{align*}
\end{Remark}

\begin{Theorem}[Wavelet basis for $\Frloc$]
\label{rlocwavelet}
Let $\Om$ be an arbitrary domain in $\R^n$ with $\Om\neq \R^n$. Let 
\begin{align*}
 0<p<\infty, 0<q<\infty, s>\sigma_{p,q} \text{ and } u>s. 
\end{align*}
and furthermore let $v>0$ such that
\begin{align*}
\max(1,p)<v < \infty, \quad s-\frac{n}{p}>-\frac{n}{v}.
\end{align*}
Then there is an orthonormal u-wavelet basis 
\begin{align*}
\Phi=\left\{\Phi_r^j: j \in \mathbb{N}_0, r=1, \ldots, N_j\right\} \subset C^u (\Om)
\end{align*}
in $L_2(\Om)$ according to \cite[Definition 2.31]{Tri08} such that the following holds:
An element $f \in L_v(\Om)$ belongs to $\Frloc$ if, and only if, $f$ can be represented as
\begin{align}
 \label{represent}
 f=\sum_{j=0}^{\infty}\sum_{r=1}^{N_j} \lambda_r^j(f) 2^{-\frac{jn}{2}}\Phi_r^j, \quad \lambda \in \fO.
 \end{align}
The representation \eqref{represent} is unique and it holds
\begin{align*}
 \lambda_r^j(f)= 2^{jn/2} (f,\Phi_r^j), \quad 
\|f|\Frloc\| \sim \| \lambda(f)|\fO\|.
\end{align*}
\end{Theorem}

%%%%%%%%%%%%%%%%%%%%%

\section{Decomposition theorems for function spaces on domains}
\label{Zerlegungsth2}
\subsection{Basic notation}
\label{Zerlegungsth} Let $n \in \N$ with $l<n$. Let $\R^n=\R^{l} \times \R^{n-l}$ and $x=(y,z) \in \R^n$,
\begin{align*}
 y=(y_1,\ldots,y_l) \in \R^l, z=(z_1,\ldots,z_{n-l}) \in \R^{n-l}.
\end{align*}
We identify $\R^l$ with the plane $\{z=0\} \subset \R^n$. Hence, in our understanding 
\begin{align*}
\hyp=\left\{x=(y,z) \in \R^n: z \neq 0\right\}.   
\end{align*}
Furthermore, let 
\begin{align*}
 Q_l = \{ x=(y,z) \in \R^n: z=0, 0<y_m<1, m=1, \ldots, l\} \subset  \R^l
\end{align*}
be the unit cube in this plane and let
\begin{align*}
 Q_l^n = \{ x=(y,z) \in \R^n: (y,0) \in Q_l, z \in \R^{n-l}\}.	
\end{align*}
Let
\begin{align*}
 \N_l^n=\left\{ \alpha=(\alpha_1,\ldots,\alpha_n) \in \N_0^n: \alpha_1=\ldots=\alpha_l=0\right\}.
\end{align*}
Then by $D^{\alpha}f$ with $\alpha \in \N_l^n$
we denote the derivatives perpendicular to $\R^l$.

\subsection{Reinforced spaces for $\R^n \setminus \R^l$}
\label{Reinf}
In \cite[Section 6.1.4]{Tri08} Triebel showed the following crucial property which paved the way to the wavelet characterization for the cube $Q$:
\begin{Proposition}[Triebel]
\label{decomptri}
 Let $l \in \N$ and $l<n$. Let
\begin{align*}
1\leq p < \infty, 0<q<\infty, 0<s-\frac{n-l}{p} \notin \N \text{ and } 
 r=\lfloor s-\frac{n-l}{p}\rfloor
\end{align*}
Then $D(Q_l^n \setminus Q_l)$ is dense in 
\begin{align*}
 \left\{f \in \Ft[Q_l^n]: \tr_l^r f=0 \right\}.
\end{align*}
\end{Proposition}
Here $\tr_l^r f$ is the trace operator onto $Q_l$. However, when $s-\frac{n-l}{p} \in \N_0$, Proposition \ref{decomptri} cannot be proven in this way and should not be true in general. As suggested in \cite[Section 6.2.3]{Tri08} we have to ``reinforce'' the function spaces $\FR$. To simplify notation in the upcoming substitute of Proposition \ref{decomptri} we replace $Q_l$ by $\R^n \setminus \R^l$ and $Q_l^n$ by $\R^n$. The basic observations remain the same.
 
\subsubsection{Hardy inequalities at $l$-dimensional planes}
We start with some basic observations regarding Hardy inequalities at planes $\R^l$ for function spaces on $\R^n$. The main observation of this section is the difference of the behaviour at $\R^l$ of $f \in \FR$ for $0<s<\frac{n-l}{p}$ - the non-critical cases - in comparison to the behaviour for $s=\frac{n-l}{p}$ - the critical cases.
\begin{Definition}
Let
\begin{align*}
 d(x):=\dist(x,\partial \Om)= \inf\{|x-y|: y \in \partial \Om \} \text{ and } 
 \Om_{\eps}:=\left\{x \in \Om: d(x)<\eps\right\}
\end{align*}
\end{Definition}
Now we take a look at $\Om=\R^n\setminus \R^l$ using the notation $x=(x',x'') \in \R^n=\R^l \times \R^{n-l}$. Then in our special situation we have $d(x)=|x''|$.

\begin{Proposition}[Sharp Hardy inequalities - the critical case]
 \label{Hardycrit}
 Let $0<\eps<1$, $1<p<\infty$ and $0<q \leq \infty$. Let $\varkappa$ be a positive monotonically decreasing function on $(0,\eps)$. Then
\begin{align*}
 \int_{(\hyp)_{\eps}} \left|\frac{\varkappa(d(x))f(x)}{\log d(x) } \right|^p \frac{ dx}{d^{n-l}(x)} \leq c \left\|f| \FR[\frac{n-l}{p}] \right\|^p
\end{align*}
for some $c>0$ and all $f \in \FR[\frac{n-l}{p}]$ if and only if $\varkappa$ is bounded.
\end{Proposition}
\begin{proof}
 The proof is a generalization of the discussion in \cite[Section 16.6]{Tri01}. There the case $l=n-1$ is considered. One uses the one-dimensional version of the Hardy inequality (16.8) in \cite{Tri01}.
 
For the ``if-part'' let at first $1<q \leq \infty$. We now use the $(n-l)$-dimensional version of (16.8) in \cite{Tri01}. Let $x=(x',x'') \in \R^n=\R^l \times \R^{n-l}$. We fix $x' \in \R^l$ and get
\begin{align*}
 \int_{|x''|<\eps} \left|\frac{\varkappa(|x''|)|f(x',x'')}{\log |x''|} \right|^p \frac{ dx''}{|x''|^{n-l}} \lesssim \left\|f| F_{p,q}^{\frac{n-l}{p}}(\R^{n-l}) \right\|^p.
\end{align*}
We now integrate over $x' \in \R^l$ and make use of the Fubini property of $\FR$, see Proposition \ref{Fubini}. Using $d(x)=|x''|$ this shows
\begin{align}
\begin{split}
 \label{Hardyl}
 \int_{(\hyp)_{\eps}} \left|\frac{\varkappa(d(x))f(x)}{\log d(x) } \right|^p \frac{ dx}{d^{n-l}(x)} &\lesssim \left\|\big\|f| F_{p,q}^{\frac{n-l}{p}}(\R^{n-l}) \big\||L_p(\R^{l})\right\| \\
& \lesssim \|f | \FR\|.
\end{split}
\end{align}
Since for fixed $p$ with $1<p<\infty$ the spaces $F_{p,q}^{\frac{n-l}{p}}(\R^{n-l})$ are monotonic with respect to $q$, inequality \eqref{Hardyl} holds for all $1<p<\infty$ and $0<q \leq \infty$.

For the ``only if-part'' we have to show that $\varkappa$ must be bounded. The proof is a generalization of the discussion in \cite[Section 16.6]{Tri01} for dimension $l=n-1$. We consider the set
\begin{align*}
 S_J^l=\{x=(x',x'') \in \R^{l} \times \R^{n-l}: |x'|<1, |x''|<2^{-J} \}, \ J \in \N
\end{align*}
and $S_J^{l,*}=S_J^l \setminus S_{J+1}^l$. We will construct an $(n-l)$-dimensional substitute of $f_J$ from (16.29) in \cite{Tri01}. We want to have $f_J \in F_{p,q}^{n-l}(\R^n)$, 
\begin{align}
\label{fJ}
 f_J(x) = J^{\frac{1}{p'}} \text{ for } x \in S_J^l \text{ and } \|f_J|\FR[\frac{n-l}{p}]\| \lesssim 1.
\end{align}
If such a sequence of $f_J$'s exist, we get
\begin{align*}
 \int_{(\hyp)_{\eps}} \left|\frac{\varkappa(d(x))f_J(x)}{\log d(x) } \right|^p \frac{ dx}{d^{n-l}(x)} &\gtrsim  \varkappa(2^{-J})^p J^{1-p}  \int_{S_J^{l,*}} \left|\frac{1}{\log d(x) } \right|^p \frac{ dx}{d^{n-l}(x)} \\
&\sim\varkappa(2^{-J})^p J^{1-p}\int_{2^{-J-1}}^{2^{-J}} \hspace{-0.15cm} r^{n-l-1} \frac{1}{r^{n-l}|\log r|^p} \ dr \\
&\gtrsim \varkappa(2^{-J})^p
\end{align*}
using $(n-l)$-dimensional spherical coordinates and $p>1$. Since the constants do not depend on $J \in \N$, this shows $\varkappa \lesssim 1$ keeping in mind $\|f_J|\FR[\frac{n-l}{p}]\| \lesssim 1$. We can define a series of $f_J$'s in the following way: For every $j \in \N$ we choose lattice points $x^{j,k} \in S_j^{l,*}$ for $k \in \{1,\ldots, C_j\}$ such that
\begin{align}
\label{overlap2}
 S_j^{l,*} \subset \bigcup_{k=1}^{C_j} B_{2^{-j}}(x^{j,k})
\end{align}
and $|x^{j,k}-x^{j,k'}| \geq 2^{-j}$ for $k \neq k'$. By a simple volume argument we have 
\begin{align}
\label{overlap1}
C_j \sim \frac{|S_j^{l,*}|}{2^{-jn}} \sim 2^{jl}.
\end{align}
Let $\psi \in \Sc(\R^n)$ be non-negative, $\psi(x)=1$ for $|x|\leq\frac{1}{2}$ and $\psi(x)=0$ for all $|x|\geq 1$. We set 
\begin{align*}
 f_J(x):=J^{-\frac{1}{p}} \sum_{j=1}^{J} \sum_{k=1}^{C_j}2^{-j \frac{l}{p}} \left[2^{j \frac{l}{p}} \psi(2^{j-1}(x-x^{j,k}))\right].
\end{align*}
At least when $q\geq 1$ and no moment conditions are necessary the functions 
\begin{align*}
\left[2^{j \frac{l}{p}} \psi(2^{j-1}(x-x^{j,k}))\right]
\end{align*}
are correctly normalized atoms in $\FR[\frac{n-l}{p}]$, see  \cite[Definition 3.1]{Sch13a}. Furthermore, by the support properties we can use the arguments in \cite[Section 2.15]{Tri01} using a modification of the sequence space. The slight overlapping of the functions $\psi(2^{j-1}(x-x^{j,k}))$ for different $j$ can be neglected. Hence by the atomic representation Theorem, see  \cite[Theorem 3.12]{Sch13a}, and $\eqref{overlap1}$ we have
\begin{align*}
 \|f_J|\FR[\frac{n-l}{p}]\| \lesssim J^{-\frac{1}{p}} \left(\sum_{j=1}^{J}\sum_{k=1}^{2^{jl}} \left(2^{-j \frac{l}{p}}\right)^p \right)^{\frac{1}{p}} \sim 1.
\end{align*}
On the other hand using $\eqref{overlap2}$ and the support properties of $\psi$ we get
\begin{align}
\label{overlap3}
 f_J(x) \geq J^{-\frac{1}{p}} \sum_{j=1}^J 1 = J^{\frac{1}{p'}} \text{ for } x \in S_J^l.
\end{align}
For $0<q<1$ one has to modify the functions $f_J$ to get moment conditions. These modifications are described in Step 5 of the proof of Theorem 13.2 in \cite{Tri01}. Then one has to define \eqref{overlap2} such that the functions $\psi(2^{j-1}(x-x^{j,k}))$ have disjoint support for fixed $j$ and different $k$. Then they cannot satisfy \eqref{overlap3}. But this is not necessary - it suffices to have $f_J(x)\geq J^{\frac{1}{p'}}$ on a set $A_J^l \subset S_J^l$ with $|A_J^l| \sim |S_J^l|$. This is possible.
\end{proof}

\begin{Proposition}[Sharp Hardy inequalities - the subcritical case]
 \label{Hardysubcrit}
Let $0<\eps<1$, $1 \leq p<\infty$ and $0<q \leq \infty$. Let $ 0<s<\frac{n-l}{p}$
and $\varkappa$ be a positive monotonically decreasing function on $(0,\eps)$. Then
\begin{align*}
 \int_{(\hyp)_{\eps}} \left|\varkappa(d(x))f(x) \right|^p \frac{ dx}{d^{sp}(x)} \leq c \left\|f| \FR \right\|^p
\end{align*}
for some $c>0$ and all $f \in \FR$ if and only if $\varkappa$ is bounded.
\end{Proposition}
\begin{proof}
The ``if-part'' can be handled in the same way as in Proposition \ref{Hardycrit} before. Now we use the $(n-l)$-dimensional version of $(16.15)$ in \cite{Tri01} having in mind $s-\frac{n-l}{p}=-\frac{n}{r}$. Here $p=1$ is allowed. Then we integrate over $x' \in \R^l$ and make use of the Fubini property \ref{Fubini} to get the desired result, using $d(x)=|x''|$.

For the ``only if-part'' we argue similar to (15.11) of \cite{Tri01}. We take
\begin{align*}
 f_j := 2^{j(-s+\frac{n}{p}-\frac{n}{2})} \cdot \Phi_{r}^j
\end{align*}
for $j \in \N$ with a wavelet $\Phi_{r}^j$ choosen from an oscillating $u$-Riesz basis for $\R^n$, see \cite[Theorem 1.20]{Tri08}, such that
\begin{align}
\label{wavedist}
 \dist(supp \ \Phi_{r}^j, \R^l) \sim 2^{-j},
\end{align}
for instance choose $m=(0,\ldots,0,1,1,\ldots,1)$ where the first $l$ coordinates are $0$. Obviously,
\begin{align*}
 f_j(x) = 2^{-j(s-\frac{n}{p})} \Phi_{r'}^0(2^j x)
\end{align*}
for a suitable $r' \in \Z^n$. Then by the atomic representation Theorem, see \cite[Theorem 3.12]{Sch13a}, we have $\|f_j|\FR\| \lesssim 1$. By \eqref{wavedist} we have for large $j$
\begin{align*}
 \int_{(\hyp)_{\eps}} \left|f_j(x) \right|^p \frac{ dx}{d^{sp}(x)} &\gtrsim 2^{-jp(s-\frac{n}{p})} \cdot \int_{supp \ \Phi_{r}^j} \frac{ dx}{d^{sp}(x)} \gtrsim 1. 
\end{align*}
Hence $\varkappa(t)$ must be bounded for $t\rightarrow 0$.
\end{proof}

\subsubsection{Definition of reinforced function spaces $\Frinf[\hyp]$}
Propositions \ref{Hardycrit} and \ref{Hardysubcrit} describe the different behaviour of the spaces $\FR[\frac{n-l}{p}]$ and $\FR$ for $0<s<\frac{n-l}{p}$ in terms of Hardy inequalities. For the space $\FR[\frac{n-l}{p}]$ we have a weaker inequality with a additional $\log$-term. This leads to the following definition of the reinforced spaces for $\Om=\hyp$ with $\partial \Om=\R^l$.

\begin{Definition}
\label{reinforcedhyp}
Let $1\leq p < \infty$, $0<q \leq \infty$ and $s> 0$. 

(i) Let $s-\frac{n-l}{p} \notin \N_0$. Then
\begin{align*}
 \Frinf[\hyp] :=\FR.
\end{align*}

(ii) Let $s-\frac{n-l}{p}=r \in \N_0$. Then
\begin{multline*}
  \Frinf[\hyp]\\
:=\left\{ f \in \FR: d^{-\frac{n-l}{p}} \cdot  D^{\alpha} f \in L_p((\hyp)_{\eps}) \text{ for all } \alpha \in \N_l^n, |\alpha|=r \right\}.
\end{multline*}
 \end{Definition}
\begin{Remark}
 For $s-\frac{n-l}{p}=r \in \N_0$ this space can be normed by
\begin{align*}
 \|f|\Frinf[\hyp]\|&:= \|f|\FR\| + \sum_{\underset{|\alpha|=r}{\alpha \in \N_l^n}}  \left(\int_{(\hyp)_{\eps}}  |D^{\alpha} f(x)|^p  \frac{ dx}{d^{n-l}(x)} \right)^{\frac{1}{p}}.
\end{align*}

\end{Remark}

\begin{Remark}
\label{indeps}
 The space $\Frinf[\hyp]$ does not depend on the choice of $\eps$ in the sense of equivalent norms since 
for $|\alpha|=r$ we have $s-r>0$ and hence
\begin{align*}
 D^{\alpha} f \in F_{p,q}^{s-r}(\R^n) \subset L_p(\R^n). 
\end{align*}
Furthermore, we can replace $d(x)$ by $\delta(x)=\min(d(x),1))$.
\end{Remark}
\begin{Remark}
 This definition is adapted by Definition 6.44 in \cite{Tri08}, where the case of a $C^{\infty}$-domain $\Om$ is considered and in this sense $l=n-1$. Then there is only one direction of derivatives to be treated - the normal derivative at $\partial \Om$.  
\end{Remark}
\begin{Remark}
Let $s-\frac{n-l}{p} \notin \N_0$. Let $f \in \FR$, $r:=\lfloor s-\frac{n-l}{p}\rfloor+1$ and additionally assume $s-r>0$: By classical properties of $\FR$ it holds
\begin{align*}
 D^{\alpha} f \in \FR[s-r] \text{ for } |\alpha|=r.
\end{align*}
Using the Hardy inequality from Proposition \ref{Hardysubcrit}  we automatically have 
\begin{align*}
 \int_{(\hyp)_{\eps}} \left|D^{\alpha} f(x) \right|^p \frac{ dx}{d^{(s-r)p}(x)} \lesssim \left\|D^{\alpha}f| \FR[s-r] \right\|^p \lesssim c \left\|f| \FR \right\|^p. 
\end{align*}
\end{Remark}

\begin{Remark}
\label{Frinfsmaller}
For the Triebel-Lizorkin spaces $\FR$ we always have
\begin{align*}
F_{p,q}^{s+\sigma}(\R^n)   \hookrightarrow \FR
\end{align*}
for $\sigma>0$. We cannot transfer such an embedding from $\FR$ to $\Frinf[\hyp]$: For incorporating the critical cases ($s-\frac{n-l}{p} \in \N_0$) we would have to show
\begin{align*}
 \|d^{-\frac{n-l}{p}} f|L_p((\hyp)_{\eps})\| \lesssim \|f|F_{p,q}^{\frac{n-l}{p}+\sigma}(\R^n)\|.
\end{align*}
Take a function $\psi \in D(\R^n)$ with $\psi(x)=1$ with $|x'|\leq 1, |x''|\leq 1$, then 
\begin{align*}
 \|d^{-\frac{n-l}{p}} \psi |L_p((\hyp)_{1})\|^p \geq  \int_{|x'|\leq 1} \int_{|x''|\leq 1} |x''|^{-(n-l)} \ dx'' \ dx'=\infty.
\end{align*}
But $\psi \in \FR$ for all $s>0$. This shows
\begin{align*}
 F_{p,q}^{\frac{n-l}{p}+\sigma}(\R^n)\lhook\joinrel\not\rightarrow F_{p,q}^{\frac{n-l}{p},\rinf}(\hyp), \quad 
  F_{p,q}^{\frac{n-l}{p},\rinf}(\hyp) \subsetneq F_{p,q}^{\frac{n-l}{p}}(\R^n). 
\end{align*}
Analogously  we have
\begin{align*}
 F_{p,q}^{r+\frac{n-l}{p}+\sigma}(\R^n) \lhook\joinrel\not\rightarrow F_{p,q}^{r+\frac{n-l}{p},\rinf}(\hyp), \quad 
 F_{p,q}^{r+\frac{n-l}{p},\rinf}(\hyp) \subsetneq F_{p,q}^{r+\frac{n-l}{p}}(\R^n).
\end{align*}
As a weaker version one can show that
\begin{align*}
 f \in F_{p,q}^{r+\frac{n-l}{p}+\sigma}(\R^n) \text{ belongs to } F_{p,q}^{r+\frac{n-l}{p},\rinf}(\hyp) 
\end{align*}
for $\sigma\in [0,1]$ if $\tr_l D^{\alpha}f= 0 \text{ for all } \alpha \in \N_l^n \text{ with } |\alpha|=r$, see \cite[Corollary 3.42]{Sch13}.
\end{Remark}

\subsection{Properties and alternative characterizations of refined localization spaces}
\label{section:refined}
Let $\Om$ be a domain with $\Om\neq \R^n$, $\Gamma=\partial \Om$, $d(x)=\dist(x,\Gamma)$ and 
 $\delta(x)=\min(d(x),1).$
\begin{Proposition}
 \label{rlocDiff}
Let $\Om$ be an arbitrary domain in $\R^n$. Let 
\begin{align*}
 0<p<\infty, 0<q<\infty, s-r>\sigma_{p,q}, \alpha \in \N^n   \text{ with }  |\alpha|=r. 
\end{align*}
It holds: If $f$ belongs to $\Frloc$, then $D^{\alpha} f$ belongs to $F_{p,q}^{s-r,\rloc}(\Om)$ with 
\begin{align*}
 \|D^{\alpha} f|F_{p,q}^{s-r,\rloc}(\Om)\|\lesssim \|f|\Frloc\| .
\end{align*}
\end{Proposition}
\begin{proof}
It suffices to prove the Proposition for $|\alpha|=1$. We will give a proof using the homogeneity Property \ref{homogen}. An alternative proof can be found using a general approach to atomic decompositions of $f \in \Frloc$, see \cite[Proposition 3.19]{Sch13}.

If $f \in \Frloc$, then $\varrho_{j,r} f \in \FR$ for $j \in \N_0, r\in \{1,\ldots,M_j\}$ and 
\begin{align*}
 D^{\alpha} (\varrho_{j,r} f)=(D^{\alpha} \varrho_{j,r}) \cdot f+\varrho_{j,r}\cdot D^{\alpha}f  \in \FR[s-1].
\end{align*}
By triangle inequality and classical differentiation properties of $\FR$ we get
\begin{align*}
 \|\varrho_{j,r} D^{\alpha} f|\FR[s-1]\| \lesssim \|(D^{\alpha} \varrho_{j,r}) \cdot f|\FR[s-1]\| + \|\varrho_{j,r} f|\FR\|.
\end{align*}
To prove the proposition, it suffices to estimate the $p$-sum of the first terms on the RHS by $\|f|\Frloc\|$. It holds
\begin{align*}
 (D^{\alpha} \varrho_{j,r}) \cdot f =  (D^{\alpha} \varrho_{j,r}) \cdot \sum_{|j-j'|\leq c}\sum_{r'} (\varrho_{j',r'} f),
\end{align*}
where $c$ and the number of summands in the sum over $r'$ are independent of $j$ and $r$, see \eqref{resun}.   

Now we make use of the homogeneity property, see Proposition \ref{homogen}, and pointwise multipliers, see \cite[Theorem 4.1]{Sch13a}. We get
\begin{align*}
 \|(D^{\alpha} &\varrho_{j,r}) \cdot f|\FR[s-1]\|  \\
&\lesssim  \sum_{|j-j'|\leq c}\sum_{r'} \|(D^{\alpha} \varrho_{j,r}) \cdot (\varrho_{j',r'} f)|\FR[s-1]\| \\
& \sim 2^{j(s-\frac{n}{p})} \sum_{|j-j'|\leq c}\sum_{r'} \|D^{\alpha} \left(\varrho_{j,r}(2^{-j}\cdot)\right) \cdot (\varrho_{j',r'} f)(2^{-j}\cdot)|\FR[s-1]\| \\
& \lesssim 2^{j(s-\frac{n}{p})} \sum_{|j-j'|\leq c}\sum_{r'} \|(\varrho_{j',r'} f)(2^{-j}\cdot)|\FR[s]\| \\
& \sim \sum_{|j-j'|\leq c}\sum_{r'} \|\varrho_{j',r'} f|\FR[s]\|,
\end{align*}
where the constants do not depend on $r$ or $j$, using property \eqref{resun}. Thus
\begin{align*}
 \sum_{j=0}^{\infty}\sum_{r=1}^{M_j}\|(D^{\alpha} \varrho_{j,r}) \cdot f|\FR[s-1]\|^p \lesssim \sum_{j=0}^{\infty}\sum_{r=1}^{M_j}\|\varrho_{j,r} \cdot f|\FR[s]\|^p=\|f|\Frloc\|^p.
\end{align*}

\end{proof}
\begin{Remark}
 For $\Om=\R^n$ there is the converse inequality
\begin{align*}
 \|f|\FR\| \leq c \sum_{|\alpha|\leq r} \|D^{\alpha} f|\FR[s-r]\|. 
\end{align*}
Such an inequality cannot hold for $\Frloc$ even on a $C^{\infty}$-domain $\Om$. For an argument see \cite[Remark 3.20]{Sch13} and for a possible weaker converse \cite[Corollary 3.44, Corollary 3.45]{Sch13}.
\end{Remark}

\begin{Proposition}
\label{rlocequi}
 Let $\Om$ be an arbitrary domain in $\R^n$ with $\Om\neq \R^n$, let 
\begin{align*}
0<p<\infty, 0< q < \infty, s>\sigma_{p,q}.
 \end{align*}
Then $f \in \Frloc$ if, and only if,
\begin{align*}
 \|f|\FO\| + \| \delta^{-s}(\cdot)f|L_p(\Om)\| < \infty.
\end{align*}
\end{Proposition}
\begin{proof}

\textit{First step:} Let $f \in \Frloc$. There is a wavelet characterization of $f$ by Theorem \ref{rlocwavelet} which leads to an atomic decomposition of $f \in \FR$, thus
\begin{align*}
 \|f|\FO\| \lesssim \|f|\Frloc\|.
\end{align*}
Furthermore, let $\varrho=\{\varrho_{j,r}\}$ be the resolution of unity adapted to the Whitney cubes $Q_{j,r}^1$. It holds
\begin{align*}
 d(x) &\sim 2^{-j} \text{ for } x \in supp \ \varrho_{j,r} \text{ for } j \in \N; \quad d(x) \gtrsim 1 \text{ for } x \in supp \ \varrho_{0,r}.
\end{align*}
We use the homogeneity property from Proposition \ref{homogen} to get
\begin{align*}
\|\delta^{-s} &\varrho_{j,r} f|L_p(\R^n)\|\sim 2^{js} \|\varrho_{j,r} f\|L_p(\R^n)\| \sim 2^{j(s-\frac{n}{p})} \|\left(\varrho_{j,r} f\right)(2^{-j}\cdot)\|L_p(\R^n)\| \\
&\lesssim 2^{j(s-\frac{n}{p})}  \|\left(\varrho_{j,r} f\right)(2^{-j}\cdot)\|\FR\| \sim \|\varrho_{j,r} f\|\FR\|  
\end{align*}
with constants independent of $j$ and $r$. 
Thus we arrive at
\begin{align*}
 \|\delta^{-s}f|L_p(\R^n)\| &\sim \left(\sum_{j,r}\| \delta^{-s} \varrho_{j,r} f|L_p(\R^n)\|^p\right)^{\frac{1}{p}} \lesssim \left(\sum_{j,r}\| \varrho_{j,r} f|\FR\|^p\right)^{\frac{1}{p}}.
\end{align*}

\textit{Second step:} Let $f \in \FO$ with $\delta^{-s}(\cdot)f \in L_p(\Om)$. Then $f \in L_v(\Om)$ for a $v>\max(1,p)$. Hence we can find a wavelet representation on $\Om$, see \cite[Theorem 2.36 ]{Tri08} in analogy to Theorem \ref{rlocwavelet}, thus
\begin{align*}
f=\sum_{j=0}^{\infty}\sum_{r=1}^{N_j} \lambda_r^j(f) 2^{-\frac{jn}{2}}\Phi_r^j
\end{align*}
with $\lambda_r^j(f) \in f_{v,2}^{0}(\Z^{\Om})$.
We split $f=f_1+f_2$, where $f_1$ collects the boundary wavelets (without moment conditions) with  
\begin{align}
\label{f1dist}
\dist(supp\ \Phi_r^{j,1},\Gamma) \sim 2^{-j}
\end{align}
and $f_2$ collects the interior wavelets (with moment conditions) with
\begin{align*}
\dist(supp\ \Phi_r^{j,2},\Gamma) \gtrsim 2^{-j}.
\end{align*}
The wavelets $\Phi_r^{j,2}$ fulfil appropriate derivative and moment conditions. Thus by  local mean Theorem 1.15 from \cite{Tri08} used for the orthogonal wavelets $\Phi_r^{j,2}$ we get
\begin{align*}
 \|\lambda_r^{j,2}(f)|\fO\|&=2^{jn/2} \| (f,\Phi_{r}^{j,2})|\fO\|=2^{jn/2} \| (\tilde{f},\Phi_{r}^{j,2})|\fO\| \\
&\lesssim \|\tilde{f}|\FR\|,
\end{align*}
where $\tilde{f}$ is an arbitrary extension of $f$ from $\Om$ to $\R^n$ (the values outside of $\Om$ do not matter for $(f,\Phi_{r}^{j,2})$). Taking the infimum over the $\FR$-norms we get
\begin{align*}
 \|\lambda_r^{j,2}(f)|\fO\| \lesssim \|f|\FO\|.
\end{align*}
Hence $f_2 \in \Frloc$ by the wavelet Theorem \ref{rlocwavelet} for $\Frloc$. By the first step
\begin{align*}
 \|\delta^{-s}f_2|L_p(\Om)\| \lesssim \|f|\FO\|.  
\end{align*}
Using triangle inequality this leads to
\begin{align}
\label{Lpd}
 \|\delta^{-s}f_1|L_p(\Om)\| \lesssim \|f|\FO\| + \|\delta^{-s}f|L_p(\Om)\|.
\end{align}
Furthermore, $\|2^{js} \lambda_r^{j,1}|f_{p,q}^{0}(\Z^{\Om})\|$ is independent of $q$ - there is a constant $C>0$ such that for all $x \in \Om$ not more than $C$ boundary wavelets are supported at $x$. This argument was also used in the proof of Theorem 2.28 in \cite{Tri08} refering to \cite[Remark 2.25]{Tri08}. Hence
\begin{align*}
  \|\lambda_r^{j,1}|f_{p,q}^{s}(\Z^{\Om})\|
&\sim \|2^{js} \lambda_r^{j,1}|f_{p,q}^{0}(\Z^{\Om})\| \sim \|2^{js} \lambda_r^{j,1}|f_{p,p}^{0}(\Z^{\Om})\| \sim \|\delta^{-s}f_1|L_p(\Om)\|
  \end{align*}
by direct calculation of the $L_p(\Om)$-norm and \eqref{f1dist}. Now, using \eqref{Lpd} we have 
\begin{align*}
  \|\lambda_r^{j,1}|f_{p,q}^{s}(\Z^{\Om})\| \lesssim \|f|\FO\| + \|\delta^{-s}f|L_p(\Om)\|,
\end{align*}
which proves that also $f_1 \in \Frloc$ by the wavelet Theorem \ref{rlocwavelet}.
\end{proof}

\subsection{Reinforced function spaces: Traces}
\label{traces}
As stated earlier Proposition \ref{decomptri} cannot hold when $r=s-\frac{n-l}{p} \in \N$. The aim of the following sections is to find a substitute. We have to care about traces at $\R^l$ for our newly introduced function spaces $\Frinf[\hyp]$ instead of $\FR$.

Let $x=(y,z) \in \R^l \times \R^{n-l}$. By $\tr_l$ we denote the trace operator 
\begin{align*}
 \tr_l: f(x) \mapsto f(y,0), \text{ for } f \in \FR
\end{align*}
on $\R^l$ (if it exists) and by  $\tr_{l}^r$ the composite map of all traces of derivatives with order not larger than $r$ and perpendicular to $\R^l$
\begin{align*}
  \tr_{l}^r: f \mapsto \left\{\tr_l D^{\alpha} f: \alpha \in \N_l^n, |\alpha|\leq r \right\}.
\end{align*}
For further informations on traces see \cite[Section 5.11]{Tri08} or {\cite[Section 4.4]{Tri92}.

\begin{Proposition}[Traces]
\label{Tra}
 Let $l \in \N_0$, $n \in \N$ with $l<n$ and $r \in \N_0$. Let $1 \leq p<\infty, 0<q<\infty$ and $
s>r + \frac{n-l}{p}.$
Then
\begin{align*}
   \tr_l^r:& \Frinf[\R^n] \rightarrow \prod_{\underset{|\alpha|\leq r}{\alpha \in \N_l^n}} F_{p,p}^{s-\frac{n-l}{p}-|\alpha|}(\R^l).
\end{align*}
\end{Proposition}
\begin{proof}
 This follows from $\Frinf[\R^n] \hookrightarrow \FR$ and Proposition 6.17 in \cite{Tri08}. The replacement of $Q_l$ by $\R_l$ is immaterial.
\end{proof}

\subsection{Decomposition theorems for $\Frinf[\hyp]$ adapted to wavelets}

Our main goal of this section is the proof of the Theorems \ref{Zerleger} and \ref{Zerlegercrit} which are the substitutes for Proposition \ref{decomptri} originating from (6.68) in \cite{Tri08}. It can be used later on for the construction of the wavelet bases on cubes.

A similar observation for the more special $C^{\infty}$-domains is the following proposition, where only traces perpendicular to the boundary $\partial \Om$ are to be considered. A proof of these results is given in \cite[Section 2.4.5]{Tri99}.

\begin{Proposition} 
\label{inftydom}
 Let $\Om$ be a bounded $C^{\infty}$-domain in $\R^n$. Let $1\leq p<\infty$, $0<q<\infty$ and 
  $0<s-\frac{1}{p} \notin \N.$ Then
\begin{align*}
 \Ft=\Frloc=\{f \in \FO: \tr_{\partial \Om}^{r} f = 0 \}.
\end{align*}
\end{Proposition}
In the following it will be easier to assume $q\geq 1$. We will give some remarks for the cases $0<q<1$ later on in Remark \ref{ZerlegerRem}.

\subsubsection{Hardy inequalities using boundary conditions at $\R^l$}
The next lemma will be a crucial observation for what follows later. It is somehow an $n$-dimensional version of the known Hardy inequality going back to \cite{Har20}, where here functions vanishing at $l$-dimensional planes are considered.
\begin{Lemma}[Hardy inequality]
 \label{HardyZerleger}
Let $n \in \N$, $l \in \N_0$, $l<n$ and $r \in \N$. Let $1\leq p <\infty$, $s>r-1+\frac{n-l}{p}$ and $d(x)$ be the distance of $x=(x',x'') \in \R^l \times \R^{n-l}$ from $\R^l$.
Then there is a constant $c>0$ such that 
\begin{align*}
 \|d^{-s}(\cdot) f|L_p(\R^n)\| \leq c \sum_{\underset{|\alpha|=r}{\alpha \in \N_l^n}} \|d^{-s+r}(\cdot) D^{\alpha}f|L_p(\R^n)\|
\end{align*}
for all $f \in C^r(\R^n)$ with $(D^{\beta}f)(x',0)=0$ for all $x' \in \R^l$ and $\beta \in \N_l^n$ with $|\beta|\leq r-1$.
\end{Lemma}
\begin{proof}
At first let us prove this lemma for $r=1$: Let $x=(x',x'') \in \R^l \times \R^{n-l}$. We fix $x',x''$ with $x''\neq 0$ and consider the one-dimensional function
\begin{align*}
 g: \R^+ \rightarrow \C, \: t \mapsto f\left(x',t \cdot \frac{x''}{|x''|}\right).
\end{align*}
Then $g(0)=f(x',0)=0$ and thus
\begin{align*}
 g(t)=&g(t)-g(0)=\int_0^{t} g'(u) \ du=\int_0^{t} \sum_{j=n-l+1}^n \frac{x_j}{|x''|} \cdot \frac{\partial f(x'',u \cdot \frac{x''}{|x''|})}{\partial x_j} \ du \\
&\leq \int_0^{t} \left|\nabla_{n-l} f\left(x'',u \cdot \frac{x''}{|x''|}\right)\right| \ du 
\end{align*}
by Cauchy's inequality. Now we apply the Hardy inequality for weighted one-dimensional $L_p$-spaces to the function $g$
\begin{align}
\label{Hardyone}
 \int_0^{\infty} \left(\frac{|f(x',t \cdot \frac{x''}{|x''|})|}{t}\right)^p \cdot t^{\alpha} \ dt \lesssim \int_0^{\infty}  \sum_{j=n-l+1}^n \left|\frac{\partial f(x',t \cdot \frac{x''}{|x''|})}{\partial x_j}\right|^p t^{\alpha} \ dt
\end{align}
for $1\leq p < \infty$ and $p>\alpha + 1$.

We integrate with respect to $(n-l)$-dimensional spherical coordinates
\begin{align}
\label{spherecord}
 \int\limits_{x'' \in \R^{n-l}} |h(x'')|^p \ dx'' = \int_{B} \tau(y) \int_0^{\infty}  t^{n-l-1}|h(ty)|^p \ dy \ dt,
\end{align}
where $B:=\{y \in \R^{n-l}: |y|=1\}$ and $\tau$ is a positive function depending only on the angle of $y$, but independent of the absolute value of $y$.

Let $h(x''):=f(x',x'') \cdot |x''|^{-s}$ for $x=(x',x'') \in \R^l \times \R^{n-l}$. Then the inner integral in \eqref{spherecord} can be estimated using \eqref{Hardyone} for every $x'' \in \R^{n-l}$. We get
\begin{align*}
 \int_0^{\infty}  t^{n-l-1}|f(x',ty)|^p t^{-sp} \ dt \lesssim \int_0^{\infty}  t^{n-l-1} \sum_{j=n-l+1}^n \left|\frac{\partial f(x',ty)}{\partial x_j}\right|^p t^{(-s+1)p} \ dt
\end{align*}
if $p \geq 1$ and $p>1+n-l-1+(-s+1)p$. The second condition is equivalent to $s>\frac{n-l}{p}$. Putting together this pointwise estimate and  \eqref{spherecord} we arrive at
\begin{align*}
 \int\limits_{x'' \in \R^{n-l}} |f(x',x'')|^p |x''|^{-sp} \ dx'' \lesssim  \int\limits_{x'' \in \R^{n-l}}  \sum_{j=n-l+1}^n \left|\frac{\partial f(x',x'')}{\partial x_j}\right|^p |x''|^{(-s+1)p} \ dx'',
\end{align*}
with constants independent of $x' \in \R^l$. Integrating over $x' \in \R^l$ finishes the lemma for  $r=1$. 

The general assertion of our lemma for arbitrary $r \in \N$ follows by mathematical induction using the same arguments for the derivatives $D^{\alpha} f$ instead of $f$ itself. Then we need $(D^{\alpha}f)(x',0)=0$ for $|\alpha|\leq r-1$ and $s>r-1+\frac{n-l}{p}$. 
\end{proof}

\begin{Remark}
 Let $1\leq p < \infty$, $1 \leq q <\infty$. In Lemma \ref{HardyZerleger} we assumed $f \in C^r(\R^n)$ with $\tr_l^{r-1} f=0$. But this lemma also holds true for $f \in \FR$ with $s=r+\frac{n-l}{p}$ and $\tr_l^{r-1} f=0$. Here is a sketch of the arguments: 
 
 Let $\R^+=\{x \in \R: x>0 \}.$ In the proof we used
 \begin{align}
 \label{inteq}
  g(t)=\int_0^t g'(u) \ du 
 \end{align}
for $g \in C^1(\R^+)$ with $g(0)=0$. We want to prove that identity \eqref{inteq} also holds true for $g^* \in \FO[\R^+]$ with $s=1+\frac{1}{p}$ and $\tr_{\{0\}} g^*=0$ (only one trace).   

For an extension $g \in \FO[\R]$ of $g^* \in \FO[\R^+]$ we find a sequence of functions $\varphi_j \in \Sc(\R)$ with $g_j:=g * \varphi_j \rightarrow g$ in $\FO[\R]$. Since $s>1$, both $g$ and its (distributional) derivative $g'$ belong to $L_p(\R)$ and hence $g * \varphi_j \rightarrow g$ and $g_j'=g' * \varphi_j \rightarrow g'$ in $L_p(\R)$. By choosing a subsequence we can assume that both sequences converge almost everywhere. Furthermore, we have $s=1+\frac{1}{p}>\frac{1}{p}$ and hence by Proposition \ref{Tra} the trace operator is continuous. This shows
\begin{align*}
 \tr_{\{0\}} g_j \rightarrow \tr_{\{0\}} g=0.
\end{align*}
Now we arrive at
\begin{align*}
  |g(t)&- \int_0^t g'(u) \ du|\\ 
  &\leq |g(t)-g_j(t)| + \left|g_j(t)- \int_0^t g_j'(u) \ du\right| + \left|\int_0^t (g_j'(u) -  g'(u)) \ du\right| \\
  &\leq |g(t)-g_j(t)| + |\tr_{\{0\}} g_j|+ c_t \|g_j'-g'|L_p(\R)\|.
\end{align*}
For almost every $t$ these three terms converge to $0$.

So, let now $l$ and $n$ be as in Lemma \ref{HardyZerleger} and (as in the proof) at first $r=1$. Then $f \in \FR$ with $s=1+\frac{n-l}{p}$ and $\tr_l f=0$ (only the trace of $f$ itself). In the proof of Lemma \ref{HardyZerleger} we constructed the function
\begin{align*}
 g_{x',x''}: \R^+ \rightarrow \C: t \mapsto f\left(x',t \cdot \frac{x''}{|x''|}\right). 
\end{align*}
But, if $f \in \FR$ with $s=1+\frac{n-l}{p}$, then $h_{x'}(x''):=f(x',x'') \in \FO[\R^{n-l}]$ for almost every $x'$ in $\R^l$ by the Fubini property \ref{Fubini}. Furthermore, using the properties of the trace operator of $F_{p,q}^{1+\frac{n-l}{p}}(\R^{n-l})$ onto one-dimensional lines (see Proposition \ref{Tra}) we get that
\begin{align*}
 g_{x',x''}: \R^+ \rightarrow \C: t \mapsto h_{x'}\left(t \cdot \frac{x''}{|x''|}\right) \in \Fpp{1+\frac{1}{p}}{\R^+}
\end{align*}
for almost all $x' \in \R^{l}$ and moreover $\tr_{\{0\}} g=0$.

Hence we have \eqref{inteq} almost everywhere. The rest of the proof of Lemma \ref{HardyZerleger} (for $r=1$) is a matter of $L_p(\R^{n-l})$-integration - as long as $f$ and $D^{\alpha} f$ belong to $L_p(\R^n)$, there are no further problems to cure.

For $r>1$ we made use of an induction argument. Hence we require not only $f \in \FR$ with $s=1+\frac{n-l}{p}$ and $\tr_l f=0$ to have \eqref{inteq}, but the same for the derivatives $D^{\alpha} f$ of $f$ with $\alpha \in \N_l^n$ upto order $|\alpha|\leq r-1$. But this is satisfied, if we assume $f \in \FR$ with $s=r+\frac{n-l}{p}$ and $\tr_l^r f=0$. Thus we have 
\end{Remark}
\begin{Corollary}
\label{HardyZerlegerC}
 Let $n \in \N$, $l \in \N_0$, $l<n$ and $r \in \N$. Let $1\leq p <\infty$, $d(x)$ be the distance of $x=(x',x'')\in \R^l \times \R^{n-l}$ from $\R^l$ and let $s=r+\frac{n-l}{p}$.
Then there is a constant $c>0$ such that 
\begin{align*}
 \|d^{-s}(\cdot) f|L_p(\R^n)\| \leq c \sum_{\underset{|\alpha|=r}{\alpha \in \N_l^n}} \|d^{-s+r}(\cdot) D^{\alpha}f|L_p(\R^n)\|
\end{align*}
for all $f \in \FR$ with $\tr_l^{r-1} f = 0$.

\end{Corollary}

\subsubsection{The decomposition theorem for the non-critical cases}
Now we come to the two main theorems of the article which pave the way to the wavelet decomposition in the non-critical and critical cases for $\Frinf[Q]$ on the cube  $Q$. Originally, Triebel proved the wavelet decomposition in \cite[Theorem 6.30]{Tri08} already for the non-critical cases, so only the critical cases are left. But, since our notation and approach is slightly different, we also give a derivation for the non-critical cases which is different from Triebel's proof.

\begin{Theorem}[The non-critical cases]
\label{Zerleger}
 Let $1 \leq p<\infty$ and $1\leq q< \infty$. Let $n \in \N$, $l \in \N_0$ and $l<n$. Let $s>0$,
\begin{align*}
 \quad s-\frac{n-l}{p} \notin \N_0 \text{ and } 
 r= \lfloor{s-\frac{n-l}{p}\rfloor}.
\end{align*}
If $r \in \N_0$, then
\begin{align}
\label{decomp1}
 \Frloc[\hyp]= \left\{f \in \Frinf[\hyp]: \tr_l^{r} f=0\right\}.
\end{align}
If $r=-1$ (hence $s<\frac{n-l}{p}$), then
\begin{align}
\label{decomp2}
 \Frloc[\hyp]=\Frinf[\hyp].
\end{align}
\end{Theorem}
\begin{proof}

 \textit{First step:} We show that $\Frloc[\hyp]$ is contained in the RHS of \eqref{decomp1} resp.\ \eqref{decomp2}. At first, if $f \in \Frloc$, then $f$ has a wavelet decomposition by Theorem \ref{rlocwavelet} and hence belongs to $\FR$ with
\begin{align}
\label{rlocest}
 \|f|\Frinf[\hyp]\|=\|f|\FR\| \lesssim \|f|\Frloc[\hyp]\|
\end{align}
by the atomic representation Theorem \cite[Theorem 3.12]{Sch13a}.

Furthermore, using \eqref{rlocest} and Remark \ref{rlocDense}, which states that $D(\hyp)$ is dense in $\Frloc[\hyp]$, we find a sequence $\{g_j\}_{j \in \N} \subset D(\hyp)$ with
\begin{align*}
 g_j \rightarrow f \text{ in } \FR
\end{align*}
for every $f \in \Frloc[\hyp]$. Hence by the continuity of the trace operator
\begin{align*}
 0=\tr_l(D^{\alpha}g_j) \rightarrow \tr_l(D^{\alpha}f) \text{ in } F_{p,p}^{s-\frac{n-l}{p}-|\alpha|}(\R^l).
\end{align*}
\textit{Second step:} We show that the RHS is contained in $\Frloc[\hyp]$. For $r=-1$, thus $0<s<\frac{n-l}{p}$, this follows from the Hardy inequalities for the subcritical case, see Proposition \ref{Hardysubcrit}, and the equivalent characterization of $\Frloc[\hyp]$ in Proposition \ref{rlocequi}.  

For the other cases ($r \in \N$, i.\,e.\ $s>\frac{n-l}{p}$) we want to give a proof using a dimension-fixing argument very similar to the proof of Lemma \ref{HardyZerleger}. 

Let $f \in \Frinf[\hyp]=\FR$ with $\tr_l^r f=0$. Let $x=(x',x'') \in \R^l \times \R^{n-l}$. We fix $x'$ and consider $g_{x'}(x'')=f(x',x'')$ as a function mapping from $\R^{n-l}$. By the Fubini property \ref{Fubini} of $\FR$ we get
\begin{align}
\label{fubint}
 \int_{x'\in \R^l} \|g(x',\cdot)|\FO[\R^{n-l}]\|^p \ dx' \lesssim \|f|\FR\|^p
\end{align}
and at least $g_{x'} \in \FO[\R^{n-l}]$ almost everywhere. Furthermore, by $\tr_l^r f=0$ we get $\tr_{\{x''=0\}} D^{\alpha} g_{x'} = 0$ for $\alpha \in \N_l^n$ with $|\alpha|\leq r$ for all $x'$ with $g_{x'} \in \FO[\R^{n-l}]$.

We now have simplified the situation: We look at a function $g \in \FO[\R^{n-l}]$ with traces at the point $x''=0$ instead of traces at an $l$-dimensional plane. If we show our theorem for this special situation, this means if we show
\begin{align}
\label{Hardypoint}
 \|d^{-s}(\cdot) g|L_p((\R^{n-l}\setminus \{0\})_{\eps})\| \lesssim \|g|\FO[\R^{n-l}]\|
\end{align}
for $g$ with $\tr_{\{0\}}^r g = 0$, then by integrating this estimate over $\R^l$ and using \eqref{fubint}, we get the desired inequality
\begin{align*}
 \|d^{-s}(\cdot) f|L_p((\hyp)_{\eps})\| \lesssim \|f|\FR\|.
\end{align*}
So, let's assume $f \in \FO[\R^{n-l}]$ and $\tr_{\{0\}}^r f =0$. Using $(n-l)$-dimensional spherical coordinates similar to \eqref{spherecord} we have
\begin{align*}
 \int_{x \in \R^n} |x|^{-sp} \cdot |f(x)|^p \ dx  = \int_B \tau(y) \int_{0}^{\infty} t^{n-l-1-sp} |f(ty)|^p \ dt \ dy,
\end{align*}
where $B:=\{y \in \R^{n-l}: |y|=1\}$ and $\tau$ is a positive function depending only on the angle $y$ of $x$ but which is independent of the absolute value $t$ of $x$.

Thus it suffices to prove
\begin{align*}
 \int_0^{\infty} t^{-(s-\frac{n-l-1}{p})p} |f(ty)|^p \ dt  \lesssim \|f|\FO[\R^{n-l}]\|.
\end{align*}
But again, this can be proven using a very special situation of our theorem, already known: If $f \in \FO[\R^{n-l}]$, then the function
\begin{align*}
 f_y: \R^+ \rightarrow \C: t \mapsto f(ty)  
\end{align*}
for $y \in B$ belongs to $F_{p,p}^{s-\frac{n-l-1}{p}}(\R^+)$ and it holds
\begin{align}
\label{traceinequi}
 \|f_y|F_{p,p}^{s-\frac{n-l-1}{p}}(\R^+)\| \leq c \ \|f|\FO[\R^{n-l}]\|
\end{align}
with a constant $c$ independent of $y \in B$: For $y=(1,0,\ldots,0)$ this follows from the trace theorem Proposition \ref{Tra}. The independency from $y \in B$ is a consequence of the rotational invariance of $\FR$. Furthermore, if $\tr_{\{0\}}^r f=0$, then $\tr_{\{0\}}^{r} f_y=0$.

Let now $f_y \in F_{p,p}^{s'}(\R^+)$ with $\tr_{\{0\}}^{r} f_y=0$ (all possible traces) and $s'=s-\frac{n-l-1}{p}$. Let $\psi \in D(\R^{+})$ be a non-negative function with $\psi(x)=1$ for $0<x\leq 1$ and $\psi(x)=0$ for $|x|\geq 2$. Then $g_y=\psi \cdot f_y \in F_{p,p}^{s'}([0,2])$ with $\tr_{\{0\}}^{r} g_y=\tr_{\{2\}}^{r} g_y=0$. 

Now we are in a one-dimensional situation. By our assumption it holds
\begin{align*}
 s-\frac{n-l}{p} \notin \N_0 \Rightarrow s'-\frac{1}{p} \notin \N_0.
\end{align*}
 By Proposition \ref{inftydom} we have $g_y \in \tilde{F\,}\!_{p,p}^{s'}([0,2])$ and by the observations in Remark \ref{Rem:Ethick} thus $g_y \in F_{p,p}^{s',\rloc}([0,2])$ with equivalent norms. Using the equivalent characterization of $F_{p,p}^{s',\rloc}([0,2])$ in Proposition \ref{rlocequi} and $\dist(t,\partial([0,2]))=t$ for $t \in (0,1)$ result in
\begin{align*}
 \int_{0}^{1} t^{-s'p} |g_y(t)|^p \ dt \lesssim \|g_y|F_{p,p}^{s',\rloc}([0,2])\| \sim \|g_y|\tilde{F\,}\!_{p,p}^{s'}([0,2])\| \sim \|g_y|F_{p,p}^{s'}([0,2])\|.
\end{align*}
This together with \eqref{traceinequi} and a pointwise multiplier argument lead to
\begin{align*}
\int_{0}^{\infty} &t^{-(s-\frac{n-l-1}{p})p} |f(ty)|^p \ dt = \int_{0}^{1} t^{-s'p} |f(ty)|^p \ dt + \int_{1}^{\infty} t^{-s'p} |f(ty)|^p \ dt \\
&\lesssim \|g_y|F_{p,p}^{s'}([0,2])\|+\|f_y|L_p(\R^+)\| \lesssim \|f_y|F_{p,p}^{s'}(\R^+)\| \lesssim \|f|\FO[\R^{n-l}]\|.
\end{align*}
\end{proof}

\subsubsection{The decomposition theorem for the critical cases}

\begin{Theorem}[The critical cases]
\label{Zerlegercrit}
 Let $1 \leq p<\infty$ and $1 \leq q< \infty$. Let $n \in \N$, $l \in \N_0$ and $l<n$. Let $s>0$ and
\begin{align*}
 r= s-\frac{n-l}{p} \in \N_0.
\end{align*}
If $r \in \N$, then
\begin{align*}
 \Frloc[\hyp]= \left\{f \in \Frinf[\hyp]: \tr_l^{r-1} f=0\right\}.
\end{align*}
If $r=0$ (hence $s=\frac{n-l}{p}$), then
\begin{align*}
 \Frloc[\hyp]=\Frinf[\hyp].
\end{align*}
\end{Theorem}
\begin{proof}
 \textit{First step:} We show, that $\Frloc[\hyp]$ is contained in the RHS. As in the first step of the proof of Theorem \ref{Zerleger}, if $f \in \Frloc[\hyp]$, then $f \in \FR$ and $\tr_l^{r-1} f =0$. 

Furthermore, by Proposition \ref{rlocDiff} it holds $D^{\alpha} f \in F_{p,q}^{\frac{n-l}{p},\rloc}(\hyp)$ for $\alpha \in \N_l^n$ with $|\alpha|=r$. Hence, by Proposition \ref{rlocequi} we have $\delta^{-\frac{n-l}{p}}(\cdot) D^{\alpha}f \in L_p(\R^n)$. Since $\delta(x) = d(x)$ for $d(x) \leq 1$, it follows $f \in \Frinf[\hyp]$.

\textit{Second step:} To show, that the RHS is contained in $\Frloc[\hyp]$ we use the equivalent characterization of $\Frloc[\hyp]$ from Proposition \ref{rlocequi}. Hence we have to prove that 
\begin{align}
\label{HardyFrinf}
\|d^{-s}(\cdot) f|L_p((\hyp)_{\eps})\| \lesssim \|f|\Frinf[\hyp]\|
 \end{align}
for all $f \in \Frinf[\hyp]$ with $\tr_l^{r-1} f=0$. If $r=0$, hence $s=\frac{n-l}{p}$ - estimate \eqref{HardyFrinf} is a direct consequence of the definition of $\Frinf[\hyp]$.

If $r>0$, then we make use of the Hardy inequality from Corollary \ref{HardyZerlegerC}: By definition of $\Frinf[\hyp]$ and $\FR \subset L_p(\R^n)$ we have $d^{-s+r} (\cdot) D^{\alpha} f \in L_p(\R^n)$ for $\alpha \in \N_l^n$ with $|\alpha|=r$. We get
\begin{align*}
 \|d^{-s}(\cdot) f|L_p(\R^n)\| \lesssim \sum_{\underset{|\alpha|=r}{\alpha \in \N_l^n}} \|d^{-s+r}(\cdot) D^{\alpha}f|L_p(\R^n)\|
\end{align*}
and $f$ belongs to $\Frloc[\hyp]= \{f \in \FR: d^{-s}(\cdot) f \in L_p((\hyp)_{\eps})\}$.
\end{proof}

\begin{Remark}
\label{ZerlegerRem}
 We want to give some remarks on the validity of Theorems \ref{Zerleger} and \ref{Zerlegercrit} if $0<q<1$:

In the non-critical cases investigated in Theorem \ref{Zerleger} the proof only makes use of $s>\sigma_{p,q}$ - then $\Frloc[\hyp]$ is defined, the Fubini property \ref{Fubini} holds and atoms do not need moment conditions.  

In the critical cases from Theorem \ref{Zerlegercrit} we used $D^{\alpha} f \in F_{p,q}^{\frac{n-l}{p},\rloc}(\hyp)$ for $\alpha \in \N_l^n$ with $|\alpha|=r$. But then naturally we have to assume $s-r=\frac{n-l}{p}> \sigma_{p,q}=\sigma_{p,q}$ by the parameters in the definition of $\Frloc$. 

But one can circumvent the direct use of $F_{p,q}^{\frac{n-l}{p},\rloc}(\hyp)$ such that it suffices to assume $s>\sigma_{p,q}$: If $f \in \Frloc[\hyp]$ with $s-r=\frac{n-l}{p}$, then by Theorem \ref{rlocwavelet} we have a wavelet decomposition of $f$ with a certain structure at the boundary $\R^l$. As in the proof of Corollary \ref{rlocequi} this gives 
\begin{align*}
 \|\delta^{-\frac{n-l}{p}}(\cdot)D^{\alpha} f|L_p(\hyp)\|  \lesssim \|f|\Frloc[\hyp]\|.
\end{align*}
For the second step it suffices to assume $s=r+\frac{n-l}{p}>\sigma_{p,q}$. Putting everything together, we can extend Theorems \ref{Zerleger} and \ref{Zerlegercrit} to $q<1$ assuming $s>\sigma_{p,q}$.  
\end{Remark}
 
\section{Outlook}
Right now Theorems \ref{Zerleger} and \ref{Zerlegercrit} seem to be relatively theoretical constructs. The important aspect is to see the similarity to the observations in \cite[Section 6.1.4]{Tri08}. Using the decomposition theorems we know that $f \in \Frinf[\hyp]$ belongs to the refined localization space $\Frloc[\hyp]$ if all existing traces on $\R^l$ are vanishing. The spaces $\Frloc[\hyp]$ have (interior) wavelet (Riesz) bases, see Theorem\ref{rlocwavelet} or the original source \cite[Theorem 2.38]{Tri08}. 

If we have an arbitrary $f \in \Frinf[\hyp]$ whose traces are not vanishing we use the technique described in \cite[Theorem 6.23]{Tri08}. We cut $f$ into
\begin{align}
\label{decompo}
f=f_1+f_2 = \left( f - (\Ext_l^{r,u} \circ\tr_l^{r}) f \right) + (\Ext_l^{r,u} \circ\tr_l^{r}) f, 
\end{align}
where $\Ext_l^{r,u}$ is the wavelet-friendly extension operator introduced in \cite[Section 6.1.3]{Tri08}. Then by construction $\tr_l^r f_1=0$ and hence $f_1$ belongs to $\Frloc[\hyp]$ by our decomposition theorems. On the other hand, $f_2$ is an extension of an element of the trace space on $\R^l$ which also admits a wavelet decomposition. Using the wavelet-friendly extension operator $\Ext_l^{r,u}$  one can extend the wavelet functions of the wavelet basis of the trace space on $\R^l$ to wavelet functions on $\R^n$. Putting both wavelet decompositions together one can decompose $f$ into wavelet-like functions and hence find a Riesz basis which is an oscillating wavelet system as Triebel defined it in \cite[Definition 2.4]{Tri08} resp.\ \cite[Definition 5.5]{Tri08}.

One can transfer this idea from reinforced Triebel-Lizorkin function spaces on $\hyp$ to reinforced Triebel-Lizorkin function spaces on the cube $Q$. Essentially one now has to consider every boundary of dimension $0$ to $n-1$ on its own, starting with dimension $0$ and caring about the traces at the boundaries using the decomposition technique \eqref{decompo} from low to high dimension. This is done in Chapter 4 of thesis \cite{Sch13} and will also be published in the future.

The most prominent exceptional space is the classical Sobolev space $W_2^1(Q)$ for the $n$-dimensional cube and $n\geq 2$ since $s-\frac{2}{p}=0$. Upto now there seems to be no wavelet (Riesz) bases in the sense of Triebel's definition for this space. With the results from Chapter 4 of thesis \cite{Sch13} we are able to show that at least a reasonable subspace of $W_2^1(Q)$, the reinforced Sobolev space $W_2^{1,\rinf}(Q)$, has a wavelet representation in this sense.

%% References


\bibliographystyle{abbrv}
\bibliography{ben}

\end{document}